\documentclass{amsart}
\usepackage{amsmath}
\usepackage{amssymb}
\usepackage{amsfonts}

\newtheorem{theorem}{Theorem}
\theoremstyle{plain}

\newtheorem{case}{Case}

\newtheorem{corollary}{Corollary}

\newtheorem{definition}{Definition}

\newtheorem{lemma}{Lemma}

\newtheorem{proposition}{Proposition}

\numberwithin{equation}{section}
\input{tcilatex}

\begin{document}
\title[]{NEW INEQUALITIES FOR Hermite-Hadamard AND SIMPSON TYPE AND
APPLICATIONS}
\author{M.E. \"{O}ZDEM\.{I}R}
\address{ATAT\"{U}RK UNIVERSITY, K. K. EDUCATION FACULTY, DEPARTMENT OF
MATHEMATICS, 25240, CAMPUS, ERZURUM, TURKEY}
\email{emos@atauni.edu.tr}
\email{yildizcetiin@yahoo.com}
\author{\c{C}ET\.{I}N YILDIZ$^{\clubsuit }$}
\thanks{$^{\clubsuit }$Corresponding Author.}

\begin{abstract}
In this paper, we obtain new bounds for the inequalities of Simpson and
Hermite-Hadamard type for functions whose second derivatives absolute values
are $P$-convex. These bounds can be much better than some obtained bounds.
Some applications for special means of real numbers are also given.
\end{abstract}

\maketitle

\section{INTRODUCTION}

Let $f:I\subset 
\mathbb{R}
\rightarrow 
\mathbb{R}
$ be a convex function defined on the interval $I$ of real numbers and a,b$%
\in I$, with $a<b$. The following inequality, known as the Hermite-Hadamard
inequality for convex functions, holds:%
\begin{equation}
f\left( \frac{a+b}{2}\right) \leq \frac{1}{b-a}\int_{a}^{b}f(x)dx\leq \frac{%
f(a)+f(b)}{2}.  \label{18}
\end{equation}

Since the inequalities in (\ref{18}) have been also known as Hadamard's
inequalities. In this work, we shall call them the Hermite-Hadamard
inequalities or H-H inequalities, for simplicity.

In recent years many authors have established several inequalities connected
to H-H\textit{\ }inequality. For recent results, refinements, counterparts,
generalizations and new\textit{\ }H-H\textit{\ }and Simpson\textit{\ }type
inequalities see the papers \cite{ALOMARI}, \cite{ssD}, \cite{DRAGOMIR}, 
\cite{usK}, \cite{LIU}, \cite{SET}, \cite{yildiz} and \cite{gsy}.

The following inequality is well known in the literature as Simpson's
inequality.

Let $f:[a,b]\rightarrow 
\mathbb{R}
$ be a four times continuously differentiable mapping on $(a,b)$ and $%
\left\Vert f^{(4)}\right\Vert _{\infty }=\underset{x\in (a,b)}{\sup }%
\left\vert f^{(4)}(x)\right\vert <\infty .$ Then, the following inequality
holds:%
\begin{equation}
\left\vert \frac{1}{3}\left[ \frac{f(a)+f(b)}{2}+2f\left( \frac{a+b}{2}%
\right) \right] -\frac{1}{b-a}\int_{a}^{b}f(x)dx\right\vert \leq \frac{1}{%
2880}\left\Vert f^{(4)}\right\Vert _{\infty }(b-a)^{2}.  \label{19}
\end{equation}

In \cite{daI}, S.S. Dragomir et.al., defined following new class of
functions.

\begin{definition}
A function $f:I\subseteq 
\mathbb{R}
\rightarrow 
\mathbb{R}
$ is $P-$function or that f belongs to the class of $P(I)$, if it is
nonnegative and for all $x,y\in I$ and $\lambda \in \lbrack 0,1]$, satisfies
the following inequality;%
\begin{equation*}
f(\lambda x+(1-\lambda )y)\leq f(x)+f(y).
\end{equation*}%
$P(I)$ contain all nonnegative monotone convex and quasi convex functions.
\end{definition}

In \cite{ahmet}, Akdemir and \"{O}zdemir defined co-ordinaded $P$-functions
and proved some inequalities and in \cite{daI}, Dragomir \textit{et al.},
proved following inequalities of Hadamard's type for $P-$functions.

\begin{theorem}
Let $f\in P(I)$, $a,b\in I$, with $a<b$ and $f\in L_{1}[a,b]$. Then the
following inequality holds.%
\begin{equation*}
f\left( \frac{a+b}{2}\right) \leq \frac{2}{b-a}\int_{a}^{b}f(x)dx\leq 2\left[
f(a)+f(b)\right] .
\end{equation*}
\end{theorem}

In \cite{selected}, Dragomir and Pearce have studied this type of
inequalities for twice differential function with bounded second derivative
and have obtained the following:

\begin{theorem}
Assume that $f:I\rightarrow R$ is continuous on $I$, twice differentiable on 
$I^{\circ }$ and there exist $k,K$ such that $k\leq f^{\prime \prime }\leq K$
on I. Then%
\begin{equation}
\frac{k}{3}\left( \frac{b-a}{2}\right) ^{2}\leq \frac{f(a)+f(b)}{2}-\frac{1}{%
b-a}\int_{a}^{b}f(x)dx\leq \frac{K}{3}\left( \frac{b-a}{2}\right) ^{2}.
\label{20}
\end{equation}
\end{theorem}

In \cite{midpoint}, Cerone and Dragomir the following integral inequality
were obtained.

\begin{theorem}
Let $f:[a,b]\rightarrow R$ be a twice differentiable mapping and suppose
that $\gamma \leq f^{\prime \prime }\leq \Gamma $\ for all $t\in (a,b)$.
Then we have%
\begin{equation}
\frac{\gamma \left( b-a\right) ^{2}}{24}\leq \frac{1}{b-a}%
\int_{a}^{b}f(x)dx-f\left( \frac{a+b}{2}\right) \leq \frac{\Gamma \left(
b-a\right) ^{2}}{24}.  \label{21}
\end{equation}
\end{theorem}

In \cite{m.zeki}, Sar\i kaya \textit{et al.} established following Lemma for
twice differentiable mappings:

\begin{lemma}
\label{yildiz2}Let $I\subset 
\mathbb{R}
$ be an open interval, with $a<b$. If $f:I\rightarrow 
\mathbb{R}
$ is a twice differentiable mapping such that $f^{\prime \prime }$ is
integrable and $0\leq \lambda \leq 1.$ Then the following identity holds:%
\begin{equation*}
\left( \lambda -1\right) f\left( \frac{a+b}{2}\right) -\lambda \frac{%
f(a)+f(b)}{2}+\frac{1}{b-a}\int_{a}^{b}f(x)dx=(b-a)^{2}\int_{0}^{1}k(t)f^{%
\prime \prime }(ta+(1-t)b)dt
\end{equation*}%
where%
\begin{equation*}
k(t)=\left\{ 
\begin{tabular}{ll}
$\frac{1}{2}t(t-\lambda ),$ & $0\leq t\leq \frac{1}{2}$ \\ 
&  \\ 
$\frac{1}{2}(1-t)(1-\lambda -t),$ & $\frac{1}{2}\leq t\leq 1.$%
\end{tabular}%
\right.
\end{equation*}
\end{lemma}

The main purpose of this paper is to point out new estimations of the (\ref%
{18}) and (\ref{19}) inequalities and to apply them in special means of the
real numbers.

\section{MAIN RESULTS}

Using Lemma \ref{yildiz2} equality we can obtain the following general
integral inequalities for $P-$convex functions.

\begin{theorem}
\label{cetin}Let $f:I\subset 
\mathbb{R}
\rightarrow 
\mathbb{R}
$ be a differentiable mapping on $I^{o}$($I^{o}$ is the interior of I), $%
a,b\in I$ with $a<b$. If $\left\vert f^{\prime \prime }\right\vert $ is $P-$%
convex function$,$ $0\leq \lambda \leq 1,$ then the following inequality
holds:%
\begin{eqnarray}
&&  \label{4} \\
&&\left\vert \left( \lambda -1\right) f\left( \frac{a+b}{2}\right) -\lambda 
\frac{f(a)+f(b)}{2}+\frac{1}{b-a}\int_{a}^{b}f(x)dx\right\vert \text{ \ \ \
\ \ }  \notag \\
&\leq &\left\{ 
\begin{tabular}{ll}
$\frac{(b-a)^{2}}{24}\left( 8\lambda ^{3}-3\lambda +1\right) \left\{
\left\vert f^{\prime \prime }(a)\right\vert +\left\vert f^{\prime \prime
}(b)\right\vert \right\} ,$ & $for$ $0\leq \lambda \leq \frac{1}{2}$ \\ 
&  \\ 
$\ \frac{(b-a)^{2}}{24}\left( 3\lambda -1\right) \left\{ \left\vert
f^{\prime \prime }(a)\right\vert +\left\vert f^{\prime \prime
}(b)\right\vert \right\} $ $\ \ \ \ \ \ ,$ & $for$ $\frac{1}{2}\leq \lambda
\leq 1.$%
\end{tabular}%
\right.  \notag
\end{eqnarray}
\end{theorem}

\begin{proof}
From Lemma 1, we have%
\begin{eqnarray}
&&  \label{1} \\
&&\left\vert \left( \lambda -1\right) f\left( \frac{a+b}{2}\right) -\lambda 
\frac{f(a)+f(b)}{2}+\frac{1}{b-a}\int_{a}^{b}f(x)dx\right\vert \text{ } 
\notag \\
&\leq &\frac{(b-a)^{2}}{2}\left[ \int_{0}^{\frac{1}{2}}\left\vert
t(t-\lambda )\right\vert \left\vert f^{\prime \prime }(ta+(1-t)b)\right\vert
dt\right.  \notag \\
&&\left. +\int_{\frac{1}{2}}^{1}\left\vert (1-t)(1-\lambda -t)\right\vert
\left\vert f^{\prime \prime }(ta+(1-t)b)\right\vert dt\right] .  \notag
\end{eqnarray}%
We assume that $0\leq \lambda \leq \frac{1}{2},$ then using the $P-$%
convexity of $\left\vert f^{\prime \prime }\right\vert ,$ we have \ \ \ \ \
\ \ \ \ \ \ \ \ \ \ \ \ \ \ \ \ \ \ \ \ \ \ \ \ \ \ \ \ \ \ \ \ \ \ \ \ \ \
\ \ \ \ \ \ \ \ \ \ \ \ \ \ \ \ 
\begin{eqnarray}
&&  \label{2} \\
&&\int_{0}^{\frac{1}{2}}\left\vert t(t-\lambda )\right\vert \left\vert
f^{\prime \prime }(ta+(1-t)b)\right\vert dt  \notag \\
&=&\int_{0}^{\lambda }t(\lambda -t)\left\vert f^{\prime \prime
}(ta+(1-t)b)\right\vert dt+\int_{\lambda }^{\frac{1}{2}}t(t-\lambda
)\left\vert f^{\prime \prime }(ta+(1-t)b)\right\vert dt  \notag \\
&\leq &\left\{ \left\vert f^{\prime \prime }(a)\right\vert +\left\vert
f^{\prime \prime }(b)\right\vert \right\} \left[ \int_{0}^{\lambda
}t(\lambda -t)dt+\int_{\lambda }^{\frac{1}{2}}t(t-\lambda )dt\right]  \notag
\\
&=&\left\{ \left\vert f^{\prime \prime }(a)\right\vert +\left\vert f^{\prime
\prime }(b)\right\vert \right\} \left( \frac{\lambda ^{3}}{3}-\frac{\lambda 
}{8}+\frac{1}{24}\right) .  \notag
\end{eqnarray}%
Similarly, we write%
\begin{eqnarray}
&&  \label{3} \\
&&  \notag \\
&&\int_{\frac{1}{2}}^{1}\left\vert (1-t)(1-\lambda -t)\right\vert \left\vert
f^{\prime \prime }(ta+(1-t)b)\right\vert dt  \notag \\
&=&\int_{\frac{1}{2}}^{1-\lambda }(1-t)(1-\lambda -t)\left\vert f^{\prime
\prime }(ta+(1-t)b)\right\vert dt  \notag \\
&&+\int_{1-\lambda }^{1}(1-t)(t+\lambda -1)\left\vert f^{\prime \prime
}(ta+(1-t)b)\right\vert dt  \notag \\
&\leq &\left\{ \left\vert f^{\prime \prime }(a)\right\vert +\left\vert
f^{\prime \prime }(b)\right\vert \right\} \left[ \int_{\frac{1}{2}%
}^{1-\lambda }(1-t)(1-\lambda -t)dt+\int_{1-\lambda }^{1}(1-t)(t+\lambda
-1)dt\right]  \notag \\
&=&\left\{ \left\vert f^{\prime \prime }(a)\right\vert +\left\vert f^{\prime
\prime }(b)\right\vert \right\} \left( \frac{2(1-\lambda )^{3}}{3}+\lambda
(1-\lambda )^{2}+\frac{7\lambda }{8}-\frac{5}{8}\right) .  \notag
\end{eqnarray}%
Using (\ref{2}) and (\ref{3}) in (\ref{1}), we see that first inequality of (%
\ref{4}) holds.

On the other hand, let $\frac{1}{2}\leq \lambda \leq 1,$ then, from $P-$%
convexity of $\left\vert f^{\prime \prime }\right\vert $ we have%
\begin{eqnarray*}
&&\int_{0}^{\frac{1}{2}}\left\vert t(t-\lambda )\right\vert \left\vert
f^{\prime \prime }(ta+(1-t)b)\right\vert dt \\
&&+\int_{\frac{1}{2}}^{1}\left\vert (1-t)(1-\lambda -t)\right\vert
\left\vert f^{\prime \prime }(ta+(1-t)b)\right\vert dt \\
&\leq &\left\{ \left\vert f^{\prime \prime }(a)\right\vert +\left\vert
f^{\prime \prime }(b)\right\vert \right\} \left[ \int_{0}^{\frac{1}{2}%
}t(\lambda -t)dt+\int_{\frac{1}{2}}^{1}(1-t)(t+\lambda -1)dt\right] \\
&=&\left\{ \left\vert f^{\prime \prime }(a)\right\vert +\left\vert f^{\prime
\prime }(b)\right\vert \right\} \left( \frac{\lambda }{4}-\frac{1}{12}%
\right) .
\end{eqnarray*}%
This is second inequality of (\ref{4}). This also completes the proof.
\end{proof}

\begin{theorem}
\label{yildiz}Let $f:I\subset 
\mathbb{R}
\rightarrow 
\mathbb{R}
$ be a differentiable mapping on $I^{o}$ , $a,b\in I$ with $a<b$. If $%
\left\vert f^{\prime \prime }\right\vert ^{q}$ is $P-$convex function$,$ $%
0\leq \lambda \leq 1$ and $q\geq 1,$ then the following inequality holds:%
\begin{eqnarray}
&&  \label{15} \\
&&\left\vert \left( \lambda -1\right) f\left( \frac{a+b}{2}\right) -\lambda 
\frac{f(a)+f(b)}{2}+\frac{1}{b-a}\int_{a}^{b}f(x)dx\right\vert  \notag \\
&\leq &\left\{ 
\begin{array}{cc}
\begin{array}{c}
\\ 
\frac{(b-a)^{2}}{48}\left( 8\lambda ^{3}-3\lambda +1\right) \left( \left\{
\left\vert f^{\prime \prime }(a)\right\vert ^{q}+\left\vert f^{\prime \prime
}(b)\right\vert ^{q}\right\} \right) ^{\frac{1}{q}},%
\end{array}
& for\text{ }0\leq \lambda \leq \frac{1}{2} \\ 
&  \\ 
\text{ \ \ \ \ \ }%
\begin{array}{c}
\frac{(b-a)^{2}}{48}\left( 3\lambda -1\right) \left( \left\{ \left\vert
f^{\prime \prime }(a)\right\vert ^{q}+\left\vert f^{\prime \prime
}(b)\right\vert ^{q}\right\} \right) ^{\frac{1}{q}} \\ 
\end{array}%
, & for\text{ }\frac{1}{2}\leq \lambda \leq 1%
\end{array}%
\right.  \notag
\end{eqnarray}
\end{theorem}

\begin{proof}
From Lemma 1 and using well known power mean inequality, we get%
\begin{eqnarray}
&&  \label{5} \\
&&\left\vert \left( \lambda -1\right) f\left( \frac{a+b}{2}\right) -\lambda 
\frac{f(a)+f(b)}{2}+\frac{1}{b-a}\int_{a}^{b}f(x)dx\right\vert \text{ } 
\notag \\
&\leq &\frac{(b-a)^{2}}{2}\left[ \int_{0}^{\frac{1}{2}}\left\vert
t(t-\lambda )\right\vert \left\vert f^{\prime \prime }(ta+(1-t)b)\right\vert
dt\right.  \notag \\
&&\left. +\int_{\frac{1}{2}}^{1}\left\vert (1-t)(1-\lambda -t)\right\vert
\left\vert f^{\prime \prime }(ta+(1-t)b)\right\vert dt\right]  \notag \\
&\leq &\frac{(b-a)^{2}}{2}\left( \int_{0}^{\frac{1}{2}}\left\vert
t(t-\lambda )\right\vert dt\right) ^{1-\frac{1}{q}}\left( \int_{0}^{\frac{1}{%
2}}\left\vert t(t-\lambda )\right\vert \left[ \left\vert f^{\prime \prime
}(ta+(1-t)b)\right\vert \right] ^{q}dt\right) ^{\frac{1}{q}}  \notag \\
&&+\left( \int_{\frac{1}{2}}^{1}\left\vert (1-t)(1-\lambda -t)\right\vert
dt\right) ^{1-\frac{1}{q}}\left( \int_{\frac{1}{2}}^{1}\left\vert
(1-t)(1-\lambda -t)\right\vert \left[ \left\vert f^{\prime \prime
}(ta+(1-t)b)\right\vert \right] ^{q}dt\right) ^{\frac{1}{q}}.  \notag
\end{eqnarray}%
Let $0\leq \lambda \leq \frac{1}{2}.$ Since $\left\vert f^{\prime \prime
}\right\vert $ is $P-$convex on [a,b], we write%
\begin{eqnarray}
&&  \label{6} \\
&&\int_{0}^{\frac{1}{2}}\left\vert t(t-\lambda )\right\vert \left[
\left\vert f^{\prime \prime }(ta+(1-t)b)\right\vert \right] ^{q}dt  \notag \\
&=&\int_{0}^{\lambda }t(\lambda -t)\left[ \left\vert f^{\prime \prime
}(ta+(1-t)b)\right\vert \right] ^{q}dt+\int_{\lambda }^{\frac{1}{2}%
}t(t-\lambda )\left[ \left\vert f^{\prime \prime }(ta+(1-t)b)\right\vert %
\right] ^{q}dt  \notag \\
&\leq &\left\{ \left\vert f^{\prime \prime }(a)\right\vert ^{q}+\left\vert
f^{\prime \prime }(b)\right\vert ^{q}\right\} \left[ \int_{0}^{\lambda
}t(\lambda -t)dt+\int_{\lambda }^{\frac{1}{2}}t(t-\lambda )dt\right]  \notag
\\
&=&\left\{ \left\vert f^{\prime \prime }(a)\right\vert ^{q}+\left\vert
f^{\prime \prime }(b)\right\vert ^{q}\right\} \left( \frac{\lambda ^{3}}{3}-%
\frac{\lambda }{8}+\frac{1}{24}\right) ,  \notag
\end{eqnarray}

\begin{eqnarray}
&&  \label{7} \\
&&\int_{\frac{1}{2}}^{1}\left\vert (1-t)(1-\lambda -t)\right\vert \left[
\left\vert f^{\prime \prime }(ta+(1-t)b)\right\vert \right] ^{q}dt  \notag \\
&=&\int_{\frac{1}{2}}^{1-\lambda }(1-t)(1-\lambda -t)\left[ \left\vert
f^{\prime \prime }(ta+(1-t)b)\right\vert \right] ^{q}dt  \notag \\
&&+\int_{1-\lambda }^{1}(1-t)(t+\lambda -1)\left[ \left\vert f^{\prime
\prime }(ta+(1-t)b)\right\vert \right] ^{q}dt  \notag \\
&\leq &\left\{ \left\vert f^{\prime \prime }(a)\right\vert ^{q}+\left\vert
f^{\prime \prime }(b)\right\vert ^{q}\right\} \left[ \int_{\frac{1}{2}%
}^{1-\lambda }(1-t)(1-\lambda -t)dt+\int_{1-\lambda }^{1}(1-t)(t+\lambda
-1)dt\right]  \notag \\
&=&\left\{ \left\vert f^{\prime \prime }(a)\right\vert ^{q}+\left\vert
f^{\prime \prime }(b)\right\vert ^{q}\right\} \left( \frac{2(1-\lambda )^{3}%
}{3}+\lambda (1-\lambda )^{2}+\frac{7\lambda }{8}-\frac{5}{8}\right) , 
\notag
\end{eqnarray}

\begin{equation}
\int_{0}^{\frac{1}{2}}\left\vert t(t-\lambda )\right\vert
dt=\int_{0}^{\lambda }t(\lambda -t)dt+\int_{\lambda }^{\frac{1}{2}%
}t(t-\lambda )dt=\frac{\lambda ^{3}}{3}+\frac{1-3\lambda }{24}  \label{8}
\end{equation}%
and%
\begin{equation}
\int_{\frac{1}{2}}^{1}\left\vert (1-t)(1-\lambda -t)\right\vert dt=\int_{%
\frac{1}{2}}^{1-\lambda }(1-t)(1-\lambda -t)dt+\int_{1-\lambda
}^{1}(1-t)(t+\lambda -1)dt=\frac{\lambda ^{3}}{3}+\frac{1-3\lambda }{24}.
\label{9}
\end{equation}%
Thus, using (\ref{6})-(\ref{9}) in (\ref{5}), we obtain the first inequality
of (\ref{15}).

Now, let $\frac{1}{2}\leq \lambda \leq 1,$ then, using the $P-$convexity of $%
\left\vert f^{\prime \prime }\right\vert ^{q}$, we have%
\begin{eqnarray}
&&\int_{0}^{\frac{1}{2}}\left\vert t(t-\lambda )\right\vert \left[
\left\vert f^{\prime \prime }(ta+(1-t)b)\right\vert \right] ^{q}dt
\label{11} \\
&=&\int_{0}^{\frac{1}{2}}t(\lambda -t)\left[ \left\vert f^{\prime \prime
}(ta+(1-t)b)\right\vert \right] ^{q}dt  \notag \\
&\leq &\int_{0}^{\frac{1}{2}}t(\lambda -t)\left\{ \left\vert f^{\prime
\prime }(a)\right\vert ^{q}+\left\vert f^{\prime \prime }(b)\right\vert
^{q}\right\} dt  \notag \\
&=&\left\{ \left\vert f^{\prime \prime }(a)\right\vert ^{q}+\left\vert
f^{\prime \prime }(b)\right\vert ^{q}\right\} \left( \frac{\lambda }{8}-%
\frac{1}{24}\right) ,  \notag
\end{eqnarray}%
similarly,%
\begin{eqnarray}
&&\int_{\frac{1}{2}}^{1}\left\vert (1-t)(1-\lambda -t)\right\vert \left[
\left\vert f^{\prime \prime }(ta+(1-t)b)\right\vert \right] ^{q}dt
\label{12} \\
&=&\int_{\frac{1}{2}}^{1}(1-t)(t+\lambda -1)\left[ \left\vert f^{\prime
\prime }(ta+(1-t)b)\right\vert \right] ^{q}dt  \notag \\
&\leq &\int_{\frac{1}{2}}^{1}(1-t)(t+\lambda -1)\left\{ \left\vert f^{\prime
\prime }(a)\right\vert ^{q}+\left\vert f^{\prime \prime }(b)\right\vert
^{q}\right\} dt  \notag \\
&=&\left\{ \left\vert f^{\prime \prime }(a)\right\vert ^{q}+\left\vert
f^{\prime \prime }(b)\right\vert ^{q}\right\} \left( \frac{\lambda }{8}-%
\frac{1}{24}\right) .  \notag
\end{eqnarray}%
We also have

\begin{equation}
\int_{0}^{\frac{1}{2}}\left\vert t(t-\lambda )\right\vert dt=\int_{\frac{1}{2%
}}^{1}\left\vert (1-t)(1-\lambda -t)\right\vert dt=\frac{3\lambda -1}{24}.
\label{13}
\end{equation}%
Therefore, if we use the (\ref{11}), (\ref{12}) and (\ref{13}) in (\ref{5}),
we obtain the second inequality of (\ref{15}). This completes the proof.
\end{proof}

\begin{corollary}
In Theorem \ref{yildiz}, if we choose $\lambda =0,$ we obtain 
\begin{equation}
\left\vert \frac{1}{b-a}\int_{a}^{b}f(x)dx-f\left( \frac{a+b}{2}\right)
\right\vert \leq \frac{\left( b-a\right) ^{2}}{48}\left( \left\{ \left\vert
f^{\prime \prime }(a)\right\vert ^{q}+\left\vert f^{\prime \prime
}(b)\right\vert ^{q}\right\} \right) ^{\frac{1}{q}}  \label{16}
\end{equation}%
which similar to the left hand side of H-H inequality.
\end{corollary}

\begin{corollary}
In Theorem \ref{yildiz} we choose $\lambda =1,$ we obtain 
\begin{equation}
\left\vert \frac{f(a)+f(b)}{2}-\frac{1}{b-a}\int_{a}^{b}f(x)dx\right\vert
\leq \frac{(b-a)^{2}}{24}\left( \left\{ \left\vert f^{\prime \prime
}(a)\right\vert ^{q}+\left\vert f^{\prime \prime }(b)\right\vert
^{q}\right\} \right) ^{\frac{1}{q}}  \label{17}
\end{equation}%
which similar to the right hand side of H-H inequality.
\end{corollary}

\begin{corollary}
In Theorem \ref{yildiz}, if we choose $\lambda =\frac{1}{3},$ we obtain%
\begin{equation*}
\left\vert \frac{1}{3}\left[ \frac{f(a)+f(b)}{2}+2f\left( \frac{a+b}{2}%
\right) \right] -\frac{1}{b-a}\int_{a}^{b}f(x)dx\right\vert \leq \frac{%
(b-a)^{2}}{162}\left( \left\{ \left\vert f^{\prime \prime }(a)\right\vert
^{q}+\left\vert f^{\prime \prime }(b)\right\vert ^{q}\right\} \right) ^{%
\frac{1}{q}}
\end{equation*}%
which similar to the Simpson inequality.
\end{corollary}

Furthermore if $f^{\prime \prime }$ is bounded on $I=[a,b]$ then we have the
following corollary:

\begin{corollary}
In Corollary 1, if $\left\vert f^{\prime \prime }\right\vert \leq M,$ $M>0,$
then we have%
\begin{equation*}
\left\vert \frac{1}{b-a}\int_{a}^{b}f(x)dx-f\left( \frac{a+b}{2}\right)
\right\vert \leq M\frac{\left( b-a\right) ^{2}}{48}2^{\frac{1}{q}}.
\end{equation*}%
Since $2^{\frac{1}{q}}\leq 2$ for $q\geq 1,$ we obtain%
\begin{equation*}
\left\vert \frac{1}{b-a}\int_{a}^{b}f(x)dx-f\left( \frac{a+b}{2}\right)
\right\vert \leq M\frac{\left( b-a\right) ^{2}}{24}
\end{equation*}%
which is (\ref{21}) inequality.
\end{corollary}

\begin{corollary}
In Corollary 2, if $\left\vert f^{\prime \prime }\right\vert \leq M,$ $M>0,$
then we have%
\begin{equation*}
\left\vert \frac{f(a)+f(b)}{2}-\frac{1}{b-a}\int_{a}^{b}f(x)dx\right\vert
\leq M\frac{(b-a)^{2}}{24}2^{\frac{1}{q}}.
\end{equation*}%
Since $2^{\frac{1}{q}}\leq 2$ for $q\geq 1,$ we obtain 
\begin{equation*}
\left\vert \frac{f(a)+f(b)}{2}-\frac{1}{b-a}\int_{a}^{b}f(x)dx\right\vert
\leq M\frac{(b-a)^{2}}{12}
\end{equation*}%
which is (\ref{20}) inequality.
\end{corollary}

Now, we will discuss about which bounds better than the other.

\begin{case}
In Corollary 5, \i f we choose $K>M$, we obtain new upper bound better than (%
\ref{20}) inequality.
\end{case}

\begin{case}
In Corollary 5, \i f we choose $K=M,$ we have the same result with (\ref{20}%
) inequality.
\end{case}

\begin{case}
In Corollary 5, \i f we choose $K<M$, (\ref{20}) inequality is better than
our result.
\end{case}

\begin{corollary}
In Corollary 3, if $\left\vert f^{\prime \prime }\right\vert \leq M,$ $M>0,$
then we have%
\begin{equation*}
\left\vert \frac{1}{3}\left[ \frac{f(a)+f(b)}{2}+2f\left( \frac{a+b}{2}%
\right) \right] -\frac{1}{b-a}\int_{a}^{b}f(x)dx\right\vert \leq M\frac{%
(b-a)^{2}}{162}2^{\frac{1}{q}}.
\end{equation*}%
Since $2^{\frac{1}{q}}\leq 2$ for $q\geq 1,$ we obtain%
\begin{equation*}
\left\vert \frac{1}{3}\left[ \frac{f(a)+f(b)}{2}+2f\left( \frac{a+b}{2}%
\right) \right] -\frac{1}{b-a}\int_{a}^{b}f(x)dx\right\vert \leq M\frac{%
(b-a)^{2}}{81}.
\end{equation*}
\end{corollary}

\section{APPLICATIONS TO SPECIAL MEANS}

We now consider the means for arbitrary real numbers $\alpha ,\beta $ $%
(\alpha \neq \beta ).$ We take

\begin{enumerate}
\item $Arithmetic$ $mean:$%
\begin{equation*}
A(\alpha ,\beta )=\frac{\alpha +\beta }{2},\text{ \ }\alpha ,\beta \in 
\mathbb{R}
^{+}.
\end{equation*}

\item $Logarithmic$ $mean$:%
\begin{equation*}
L(\alpha ,\beta )=\frac{\alpha -\beta }{\ln \left\vert \alpha \right\vert
-\ln \left\vert \beta \right\vert },\text{ \ \ }\left\vert \alpha
\right\vert \neq \left\vert \beta \right\vert ,\text{ }\alpha ,\beta \neq 0,%
\text{ }\alpha ,\beta \in 
\mathbb{R}
^{+}.
\end{equation*}

\item $Generalized$ $log-mean$:%
\begin{equation*}
L_{n}(\alpha ,\beta )=\left[ \frac{\beta ^{n+1}-\alpha ^{n+1}}{(n+1)(\beta
-\alpha )}\right] ^{\frac{1}{n}},\text{ \ \ \ }n\in 
\mathbb{Z}
\backslash \{-1,0\},\text{ }\alpha ,\beta \in 
\mathbb{R}
^{+}.
\end{equation*}
\end{enumerate}

Now using the results of Section 2, we give some applications for special
means of real numbers.

\begin{proposition}
Let $a,b\in 
\mathbb{R}
,$ $0<a<b$ and $n\in 
\mathbb{Z}
,$ $\left\vert n(n-1)\right\vert \geq 3,$ then, for all $q\geq 1,$ the
following inequality holds:%
\begin{equation*}
\left\vert L_{n}^{n}(a,b)-A^{n}(a,b)\right\vert \leq \left\vert
n(n-1)\right\vert \frac{(b-a)^{2}}{48}\left( \left\{
a^{q(n-2)}+b^{q(n-2)}\right\} \right) ^{\frac{1}{q}}.
\end{equation*}
\end{proposition}

\begin{proof}
The proof is obvious from Corollary 4 applied to the $P$-convex mapping $%
f(x)=x^{n},$ $x\in \lbrack a,b],$ $n\in 
\mathbb{Z}
$.
\end{proof}

\begin{proposition}
Let $a,b\in 
\mathbb{R}
,$ $0<a<b$ and $n\in 
\mathbb{Z}
,$ $\left\vert n(n-1)\right\vert \geq 3,$ then, for all $q\geq 1,$ the
following inequality holds:%
\begin{equation*}
\left\vert A(a^{n},b^{n})-L_{n}^{n}(a,b)\right\vert \leq \left\vert
n(n-1)\right\vert \frac{(b-a)^{2}}{24}\left( \left\{
a^{q(n-2)}+b^{q(n-2)}\right\} \right) ^{\frac{1}{q}}.
\end{equation*}
\end{proposition}

\begin{proof}
The proof is obvious from Corollary 6 applied to the $P$-convex mapping $%
f(x)=x^{n},$ $x\in \lbrack a,b],$ $n\in 
\mathbb{Z}
$.
\end{proof}

\begin{proposition}
Let $a,b\in 
\mathbb{R}
,$ $0<a<b$ and $n\in 
\mathbb{Z}
,$ $\left\vert n(n-1)\right\vert \geq 3,$ then, for all $q\geq 1,$ the
following inequality holds:%
\begin{equation*}
\left\vert \frac{1}{3}A(a^{n},b^{n})+\frac{2}{3}A^{n}(a,b)-L_{n}^{n}(a,b)%
\right\vert \leq \left\vert n(n-1)\right\vert \frac{(b-a)^{2}}{162}\left(
\left\{ a^{q(n-2)}+b^{q(n-2)}\right\} \right) ^{\frac{1}{q}}.
\end{equation*}
\end{proposition}

\begin{proof}
The proof is obvious from Corollary 8 applied to the $P$-convex mapping $%
f(x)=x^{n},$ $x\in \lbrack a,b],$ $n\in 
\mathbb{Z}
$.
\end{proof}

\end{document}